\documentclass[11pt,a4paper]{article}

\usepackage{amsmath,amssymb}
\newtheorem{theorem}{Theorem}[section]
\newtheorem{corollary}{Corollary}[section]

\newtheorem{proposition}{Proposition}[section]
\newtheorem{definition}{Definition}[section]
\newtheorem{remark}{Remark}[section]
\numberwithin{equation}{section}
\newtheorem{example}{Example}[section]
\newenvironment{proof}[1][Proof]{\noindent\textbf{#1.} }{\hfill {$\Box$}}
\numberwithin{equation}{section}

\begin{document}

\title{\textsc{Discrete Asymptotic Behaviors for Skew-Evolution Semiflows on Banach Spaces }}
\author{\textsc{Codru\c{t}a Stoica \ \ \ Mihail Megan}}
\date{}
\maketitle

{\footnotesize \noindent \textbf{Abstract.} The paper emphasizes
asymptotic behaviors, as stability, instability, dichotomy and
trichotomy for skew-evolution semiflows, defined by means of
evolution semiflows and evolution cocycles and which can be
considered generalizations for evolution operators and
skew-product semiflows. The definition are given in continuous
time, but the unified treatment for the characterization of the
studied properties in the nonuniform case is given in discrete
time.

The property of trichotomy, introduced in finite dimension in
\cite{ElHa_JMAA} as a natural generalization for the dichotomy of
linear time-varying differential systems, was studied in
continuous time and from uniform point of view in
\cite{MeSt_BSUPT} and in discrete time and from nonuniform point
of view in \cite{MeSt_TP07}, but for a particular case of
one-parameter semiflows.}

{\footnotesize \vspace{3mm} }

{\footnotesize \noindent \textit{Mathematics Subject
Classification:} 93D20, 34D05, 34D09}

{\footnotesize \vspace{2mm} }

{\footnotesize \noindent \textit{Keywords:} Evolution semiflow,
evolution cocycle, skew-evolution semiflow, exponential stability,
exponential instability, exponential dichotomy, exponential
trichotomy}

\section{Notations. Definitions}

Let us consider $(X,d)$ a metric space, $V$ a Banach space and
$\mathcal{B}(V)$ the space of all bounded linear operators from
$V$ into itself. We will denote $Y=X\times V$, $T
=\left\{(t,t_{0})\in \mathbf{R}^{2}, \ t\geq t_{0}\geq 0\right\}$
and $\Delta=\left\{ (m,n)\in \mathbb{N}^{2}, \ m\geq n\right\}$.
Let $\mathcal{R}=\{R:\mathbb{R}_{+}\rightarrow\mathbb{R}_{+}| \ R
\ \textrm{nondecreasing}, \ R(0)=0, \ R(t)>0, \ \forall t>0\}$. By
$I$ is denoted the identity operator on $V$. Let $P:Y\rightarrow
Y$ be a projector given by $P(x,v)=(x,P(x)v)$, where $P(x)$ is a
projection on $Y_{x}=\{x\}\times V$ and $x\in X$.

\begin{definition}\rm\label{def_sfl_ev}
A mapping $\varphi: T\times  X\rightarrow  X$ is called
\emph{evolution semiflow} on $ X$ if following relations hold:

$(s_{1})$ $\varphi(t,t,x)=x, \ \forall (t,x)\in
\mathbf{R}_{+}\times X$

$(s_{2})$ $\varphi(t,s,\varphi(s,t_{0},x))=\varphi(t,t_{0},x),
\forall t\geq s\geq t_{0}\geq 0, \ x\in X$.
\end{definition}

\begin{definition}\rm\label{def_aplcoc_ev}
A mapping $\Phi: T\times  X\rightarrow \mathcal{B}(V)$ is called
\emph{evolution cocycle} over an evolution semiflow $\varphi$ if:

$(c_{1})$ $\Phi(t,t,x)=I$, the identity operator on $V$, $\forall
(t,x)\in \mathbf{R}_{+}\times X$

$(c_{2})$
$\Phi(t,s,\varphi(s,t_{0},x))\Phi(s,t_{0},x)=\Phi(t,t_{0},x),\forall
t\geq s\geq t_{0}\geq 0, \ x\in X$.
\end{definition}

\begin{definition}\rm\label{def_coc_ev_1}
The mapping $C: T\times Y\rightarrow Y$ defined by the relation
\begin{equation}
C(t,s,x,v)=(\varphi(t,s,x),\Phi(t,s,x)v),
\end{equation}
where $\Phi$ is an evolution cocycle over an evolution semiflow
$\varphi$, is called \emph{skew-evolution semiflow} on $Y$.
\end{definition}

\begin{example}\label{ex_ce}\rm
Let $V=\mathbb{R}^{3}$ endowed with the norm
\begin{equation*}
\left\Vert(v_{1},v_{2},v_{3})\right\Vert=|v_{1}|+|v_{2}|+|v_{3}|.
\end{equation*}
We denote $\mathcal{C}=\mathcal{C}(\mathbb{R}_{+},\mathbb{R}_{+})$
the set of all continuous functions $x:\mathbb{R}_{+}\rightarrow
\mathbb{R}_{+}$, endowed with the topology of uniform convergence
on compact subsets of $\mathbb{R}_{+}$, metrizable relative to the
metric
\begin{equation*}
d(x,y)=\sum_{n=1}^{\infty}\frac{1}{2^{n}}\frac{d_{n}(x,y)}{1+d_{n}(x,y)},
\ \textrm{unde} \ d_{n}(x,y)=\underset{t\in
[0,n]}\sup{|x(t)-y(t)|}.
\end{equation*}
If $x\in \mathcal{C}$ then for all $t\in \mathbb{R}_{+}$ we denote
$x_{t}(s)=x(t+s)$, $x_{t}\in \mathcal{C}$. Let $ X$ the closure in
$\mathcal{C}$ of the set $\{f_{t},t\in \mathbb{R}_{+}\}$, where
$f:\mathbb{R}_{+}\rightarrow \mathbb{R}_{+}^{*}$ is a
nondecreasing function with the property $\underset{t\rightarrow
\infty}\lim{f(t)}=l>0$. Then $( X,d)$ is a metric space and the
mapping
\begin{equation*}
\varphi: T\times  X\rightarrow  X, \
\varphi(t,s,x)(\tau)=x(t-s+\tau)
\end{equation*}
is an evolution semiflow on $ X$. The mapping $\Phi: T\times
X\rightarrow \mathcal{B}(V)$, given by
\begin{equation*}
\Phi(t,s,x)v=\left(e^{-2(t-s)x(0)+\int_{0}^{t}x(\tau)d\tau}v_{1},
e^{t-s+\int_{0}^{t}x(\tau)d\tau}v_{2},e^{-(t-s)x(0)+2\int_{0}^{t}x(\tau)d\tau}v_{3}\right)
\end{equation*}
is an evolution cocycle. Then $C=(\varphi,\Phi)$ is a
skew-evolution semiflow.
\end{example}

\begin{remark}\rm
If $C=(\varphi, \Phi)$ is a skew-evolution semiflow, then
$C_{\lambda}=(\varphi, \Phi_{\lambda})$, where $\lambda \in
\mathbb{R}$ and
\begin{equation}\label{relcevshift}
\Phi_{\lambda}: T\times  X\rightarrow \mathcal{B}(V), \
\Phi_{\lambda}(t,t_{0},x)=e^{-\lambda(t-t_{0})}\Phi(t,t_{0},x),
\end{equation}
is also a skew-evolution semiflow.
\end{remark}

\section{Nonuniform discrete exponential stability}

\begin{definition}\rm\label{def_ns_nes}
A skew-evolution semiflow $C$ is said to be

$(s)$ \emph{stable} if there exists a mapping
$N:\mathbb{R}_{+}\rightarrow\mathbb{R}_{+}^{*}$ such that
\begin{equation}\label{rel_exp_stab}
\left\Vert \Phi(t,t_{0},x)v\right\Vert \leq N(s)\left\Vert
\Phi(s,t_{0},x)v\right\Vert,
\end{equation}%
for all $(t,s),(s,t_{0})\in  T$ and all $(x,v)\in Y$;

$(es)$ \emph{exponentially stable} if there exist a constant $\nu
>0$ and a mapping
$N:\mathbb{R}_{+}\rightarrow\mathbb{R}_{+}^{*}$ such that
\begin{equation}\label{rel_exp_stab}
\left\Vert \Phi(t,t_{0},x)v\right\Vert \leq N(s)e^{-\nu
(t-s)}\left\Vert \Phi(s,t_{0},x)v\right\Vert,
\end{equation}%
for all $(t,s),(s,t_{0})\in  T$ and all $(x,v)\in Y$.
\end{definition}

\begin{example}\rm\label{ex_nues1}
Let $ X=\mathbb{R}_{+}$ and $V=\mathbb{R}$. We consider the
continuous function
\[
f:\mathbb{R}_{+}\rightarrow[1,\infty), \ f(n)=e^{2n} \
\textrm{and} \ f\left(n+\frac{1}{e^{n^{2}}}\right)=1
\]
and the mapping
\[
\Phi_{f}:
T\times\mathbb{R}_{+}\rightarrow\mathcal{B}(\mathcal{\mathbb{R}}),
\ \Phi_{f}(t,s,x)v=\frac{f(s)}{f(t)}e^{-(t-s)}v.
\]
Then $C_{f}=(\varphi,\Phi_{f})$ is a skew-evolution semiflow
$Y=\mathbb{R}_{+}\times \mathbb{R}$ over any evolution semiflow
$\varphi$ pe $\mathbb{R}_{+}$. As
\[
\left | \Phi_{f}(t,s,x)v\right |\leq f(s)e^{-(t-s)}|v|, \ \forall
(t,s,x,v)\in T\times Y,
\]
it follows that $C_{f}$ is exponentially stable.
\end{example}

\begin{definition}\rm\label{def_neg}
A skew-evolution semiflow $C$ has \emph{exponential growth} if
there exist the mappings $M$,
$\omega:\mathbb{R}_{+}\rightarrow\mathbb{R}_{+}^{*}$ such that
\begin{equation}\label{rel_neg}
\left\Vert \Phi(t,t_{0},x)v\right\Vert \leq M(s)e^{\omega(s)
(t-s)}\left\Vert \Phi(s,t_{0},x)v\right\Vert,
\end{equation}%
for all $(t,s),(s,t_{0})\in  T$ and all $(x,v)\in Y$.
\end{definition}

\vspace{3mm}

In discrete time, the exponential stability can be described as
follows.

\begin{proposition}\label{caract_es_discret}
A skew-evolution semiflow $C$ with exponential growth is
exponentially stable if and only if there exist a constant $\mu>
0$ and a sequence of real numbers $(a_{n})_{n\geq 0}$ with the
property $a_{n}\geq 1, \ \forall n\geq 0$ such that
\begin{equation}
\left\Vert\Phi(n,m,x)v\right\Vert\leq a_{m}e^{-\mu(n-m)}\left\Vert
v\right\Vert,
\end{equation}
for all $(n,m)\in \Delta$ and all $(x,v)\in Y$.
\end{proposition}

\begin{proof}
\emph{Necessity}. It is obtained immediately if we consider in
relation $(\ref{rel_exp_stab})$ $t=n$, $s=t_{0}=m$ and if we
define
\[
a_{m}=N(m), \ m\in\mathbb{N} \ \textrm{and} \ \mu=\nu>0,
\]
where the existence of
$N:\mathbb{R}_{+}\rightarrow\mathbb{R}_{+}^{*}$ and of $\nu$ is
given be Definition \ref{def_ns_nes}.

\emph{Sufficiency}. As a first step, if $t\geq t_{0}+1$, we denote
$n=[t]$ and $n_{0}=[t_{0}]$. Following relations hold
\[
n\leq t<n+1, \ n_{0}\leq t_{0}<n_{0}+1, \ n_{0}+1\leq n.
\]
We obtain
\[
\left\Vert\Phi(t,t_{0},x)v\right\Vert\leq
\]
\[
\leq
M(n)e^{\omega(n)(t-n)}\left\Vert\Phi(n,n_{0}+1,\varphi(n_{0}+1,t_{0},x))\Phi(n_{0}+1,t_{0},x)v\right\Vert\leq
\]
\[
\leq a_{n}M^{2}(n)e^{2[\omega(n)+\mu]}e^{-\mu(t-t_{0})}\left\Vert
v\right\Vert,
\]
for all $(x,v)\in Y$, where functions $M$ and $\omega$ are given
by Definition \ref{def_neg}.

As a second step, for $t\in[t_{0},t_{0}+1)$ we have
\[
\left\Vert\Phi(t,t_{0},x)v\right\Vert\leq
M(t_{0})e^{\omega(t_{0})(t-t_{0})}\left\Vert v\right\Vert\leq
M(t_{0})e^{\omega(t_{0})+\mu} e^{-\mu(t-t_{0})}\left\Vert
v\right\Vert,
\]
for all $(x,v)\in Y$.

Hence, $C$ is exponentially stable.
\end{proof}

\vspace{3mm}

Some characterizations for the exponential stability of
skew-evolution semiflows indiscrete time are given by the next
results.

\begin{theorem}\label{th_D_discret_neunif}
A skew-evolution semiflow $C$ with exponential growth is
exponentially stable if and only if there exist a mapping
$R\in\mathcal{R}$, a constant $\rho>0$ and a sequence of real
numbers $(\alpha_{n})_{n\geq 0}$ with the property $\alpha_{n}\geq
1, \ \forall n\geq 0$ such that
\begin{equation}
\sum_{k=n}^{m}R\left(e^{\rho(k-n)}\left\Vert
\Phi(k,n,x)v\right\Vert\right)\leq\alpha_{n} R\left(\left\Vert
v\right\Vert\right),
\end{equation}
for all $(m,n)\in \Delta$ and all $(x,v)\in Y$.
\end{theorem}

\begin{proof} \emph{Necessity}. We consider $R(t)=t, \ t\geq 0$.
Proposition \ref{caract_es_discret} assures the existence of a
constant $\nu> 0$ and of a sequence of real numbers
$(a_{n})_{n\geq 0}$ with the property $a_{n}\geq 1, \ \forall
n\geq 0$. We obtain for $\rho=\frac{\nu}{2}>0$ and according to
Proposition \ref{caract_es_discret}
\[
\sum_{k=n}^{m}e^{\rho(k-n)}\left\Vert\Phi(k,n,x)v\right\Vert \leq
a_{n} \sum_{k=n}^{m}e^{\rho(k-n)}e^{-\nu(k-n)}\left\Vert
\Phi(n,n,x)v\right\Vert =
\]
\[
=a_{n}\left\Vert v\right\Vert
\sum_{k=n}^{m}e^{-\frac{\nu}{2}(k-n)}\leq \alpha_{n}\left\Vert
v\right\Vert, \forall m,n \in \Delta, \forall (x,v)\in Y,
\]
where we have denoted
\[
\alpha_{n}=a_{n}e^{\frac{\nu}{2}}, \ n\in \mathbb{N}.
\]
\emph{Sufficiency}. Let $t\geq t_{0}+1$, $t_{0}\geq 0$. We define
$n=[t]$ and $n_{0}=[t_{0}]$. We consider
$C_{-\rho}=(\varphi,\Phi_{-\rho})$ given as in relation
(\ref{relcevshift}). It follows that
\[
R\left(\left\Vert\Phi_{-\rho}(t,t_{0},x)v\right\Vert\right)=R\left
(\left\Vert\Phi_{-\rho}(t,n,\varphi(n,t_{0},x))\Phi_{-\rho}(n,t_{0},x)v\right\Vert\right)\leq
\]
\[
\leq
R\left(M(n)e^{\omega(t-n)}\left\Vert\Phi_{-\rho}(n,t_{0},x)v\right\Vert\right)
\leq
R\left(M(n)e^{\omega}\left\Vert\Phi_{-\rho}(n,t_{0},x)v\right\Vert\right),
\]
for all $(x,v)\in Y$, where $M$ and $\omega$ are given by
Definition \ref{def_neg}. We obtain further for $m\geq n$
\[
\int_{t_{0}+1}^{t}R\left(\left\Vert\Phi_{-\rho}(\tau,t_{0},x)v\right\Vert\right)d\tau
\leq
\sum_{k=[t_{0}]+1}^{m}R\left(M(n)e^{\omega}\left\Vert\Phi_{-\rho}(k,t_{0},x)v\right\Vert\right)
\leq
\]
\[
\leq \beta_{n}R\left(\left\Vert v\right\Vert\right),
\]
for all $(x,v)\in Y$, where $\beta_{n}=M(n)\alpha_{n}e^{\omega}$.
Then there exist $N\geq 1$ and $\nu>0$ such that
\[
\left\Vert\Phi_{-\rho}(t,t_{0},x)v\right\Vert\leq
Ne^{-\nu(t-t_{0})}\left\Vert v\right\Vert, \ \forall t\geq
t_{0}+1, \ \forall (x,v)\in Y.
\]
On the other hand, for $t\in[t_{0},t_{0}+1)$ we have
\[
R\left(\left\Vert\Phi_{-\rho}(t,t_{0},x)v\right\Vert\right)\leq
R\left(Me^{\omega(t-t_{0})}\left\Vert v\right\Vert\right)\leq
R\left(Me^{\omega}\left\Vert v\right\Vert\right), \ \forall
(x,v)\in Y.
\]
We obtain that $C_{-\rho}$ is stable, where
\[
\Phi_{-\rho}(m,n,x)=e^{\rho(m-n)}\left\Vert\Phi(m,n,x)\right\Vert,
\ (m,n,x)\in\Delta\times X.
\]
Hence, there exists a sequence $(a_{n})_{n\geq 0}$ with the
property $a_{n}\geq 1, \ \forall n\geq 0$, such that
\[
e^{\rho(m-n)}\left\Vert\Phi(m,n,x)v\right\Vert\leq a_{n}\left\Vert
v\right\Vert, \ \forall (m,n,x,v)\in\Delta\times Y,
\]
which implies the exponential stability of $C$ and ends the proof.
\end{proof}

\vspace{3mm}

If we consider $R(t)=t^{p}$, $t\geq 0$, $p>0$ we obtain

\begin{corollary}\label{cor_D_discret_neunif}
Let $p>0$. A skew-evolution semiflow $C$ is exponentially stable
if and only if there exist a constant $\rho>0$ and a sequence of
real numbers $(\alpha_{n})_{n\geq 0}$ with the property
$\alpha_{n}\geq 1, \ \forall n\geq 0$ such that
\begin{equation}
\sum_{k=n}^{m}e^{p\rho(k-n)}\left\Vert \Phi(k,n,x)v\right\Vert^
{p}\leq\alpha_{n} \left\Vert v\right\Vert^{p},
\end{equation}
for all $(m,n)\in \Delta$ and all $(x,v)\in Y$.
\end{corollary}

\begin{theorem}\label{th_B_discret_neunif}
A skew-evolution semiflow $C$ is exponentially stable if and only
if there exist a function $R\in\mathcal{R}$, a constant $\gamma>0$
a sequence of real numbers $(\beta_{n})_{n\geq 0}$ with the
property $\beta_{n}\geq 1, \ \forall n\geq 0$ such that
\begin{equation}
\sum_{k=n}^{m}R\left(e^{\gamma(m-k)}\left\Vert
\Phi(m,k,\varphi(k,n,x))^{*}v^{*}\right\Vert\right)\leq\beta_{n}
R\left(\left\Vert v^{*}\right\Vert\right),
\end{equation}
for all $(m,n)\in \Delta$ and all $(x,v^{*})\in X\times V^{*}$.
\end{theorem}

\begin{proof} \emph{Necessity}. For $R(t)=t, \ t\geq 0$
and $\gamma=\frac{\nu}{2}>0$ we obtain, according to Definition
\ref{def_ns_nes} and Proposition \ref{caract_es_discret},
\[
\sum_{k=n}^{m}e^{\frac{\nu}{2}(m-k)}\left\Vert\Phi(m,k,\varphi(k,n,x))^{*}v^{*}\right\Vert
\leq a_{n}\left\Vert
v^{*}\right\Vert\sum_{k=n}^{m}e^{-\frac{\nu}{2}(m-k)}\leq\beta_{n}
\left\Vert v^{*}\right\Vert,
\]
where we have denoted
\[
\beta_{n}=\frac{a_{n}}{1-e^{-\frac{\nu}{2}}}, \ n\in \mathbb{N},
\]
where $\nu> 0$ and the sequence of real numbers $(a_{n})_{n\geq
0}$ with the property $a_{n}\geq 1, \ \forall n\geq 0$ are given
by Definition \ref{def_ns_nes}.

\emph{Sufficiency}. Let $C_{-\gamma}=(\varphi,\Phi_{-\gamma})$ be
given as in relation (\ref{relcevshift}). For all
$(k,n_{0})\in\Delta$ we obtain
\[
R\left(\left\Vert\Phi_{-\gamma}(k,n_{0},x)v^{*}\right\Vert\right)\leq
\beta_{n}R\left(\left\Vert v^{*}\right\Vert\right), \ \forall x\in
X, \forall v^{*}\in V^{*}.
\]
Further we have
\[
\sum_{k=n_{0}}^{m}R\left(\left\Vert\Phi_{-\gamma}(m,n_{0},x)\right\Vert\right)\leq
\sum_{k=n_{0}}^{m}R\left(\left\Vert\Phi_{-\gamma}(m,k,x)\right\Vert\left\Vert\Phi_{-\gamma}(k,n_{0},x)\right\Vert\right)\leq
\]
\[
\leq\widetilde{M}
\sum_{k=n_{0}}^{m}R\left(\left\Vert\Phi_{-\gamma}(m,k,x)\right\Vert\right)\leq
\widetilde{M}^{2},
\]
which implies
\[
R\left(\left\Vert\Phi_{-\gamma}(m,n_{0},x)\right\Vert\right)
\leq\frac{\widetilde{M}^{2}}{m-n_{0}+1}, \ \forall
(m,n_{0},x)\in\Delta\times X.
\]
There exist $n_{1}>n_{0}+1$ such that
\[
\frac{\widetilde{M}^{2}}{n_{1}-n_{0}+1}\leq
R\left(\frac{1}{2}\right).
\]
We obtain that $C_{-\gamma}$ is stable. According to the proof of
Theorem \ref{th_D_discret_neunif} it follows that $C$ is
exponentially stable.
\end{proof}

\vspace{3mm}

For $R(t)=t^{p}$, $t\geq 0$, $p>0$ we obtain

\begin{corollary}\label{cor_B_discret_neunif}
Let $p>0$. A skew-evolution semiflow $C$ is exponentially stable
if and only if there exist a constant $\gamma>0$ and a sequence of
real numbers $(\beta_{n})_{n\geq 0}$ with $\beta_{n}\geq 1, \
\forall n\geq 0$ such that
\begin{equation}
\sum_{k=n}^{m}e^{p\gamma(m-k)}\left\Vert
\Phi(m,k,\varphi(k,n,x))^{*}v^{*}\right\Vert^{p}\leq\beta_{n}
\left\Vert v^{*}\right\Vert^{p},
\end{equation}
for all $(m,n)\in \Delta$ and all $(x,v^{*})\in X\times V^{*}$.
\end{corollary}

\section{Nonuniform discrete exponential instability}

\begin{definition}\rm\label{i}\label{cis}
A skew-evolution semiflow $C$ is said to be

$(is)$ \emph{instable} if there exists a mapping
$N:\mathbb{R}_{+}\rightarrow\mathbb{R}_{+}^{*}$ such that
\begin{equation}\label{rel_instab}
N(t)\left\Vert \Phi(t,t_{0},x)v\right\Vert \geq \left\Vert
\Phi(s,t_{0},x)v\right\Vert,
\end{equation}
for all $(t,s),(s,t_{0})\in  T$ and all $(x,v)\in Y$;

$(eis)$ \emph{exponentially instable} if there exist a mapping
$N:\mathbb{R}_{+}\rightarrow\mathbb{R}_{+}^{*}$ and a constant
$\nu> 0 $ such that
\begin{equation}\label{rel_exp_instab}
N(t)\left\Vert\Phi(t,t_{0},x)v\right\Vert\geq
e^{\nu(t-s)}\left\Vert\Phi(s,t_{0},x)v\right\Vert,
\end{equation}
for all $(t,s),(s,t_{0})\in T$ and all $(x,v)\in Y$.
\end{definition}

\begin{example}\rm\label{ex_nueis1}
Let $ X=\mathbb{R}_{+}$ and $V=\mathbb{R}$. We consider the
function
\[
f:\mathbb{R}_{+}\rightarrow[1,\infty), \ f(n)=1 \ \textrm{and} \
f\left(n+\frac{1}{e^{n^{2}}}\right)=e^{2n}
\]
and the mapping
\[
\Phi_{f}: T\times\mathbb{R}_{+}\rightarrow\mathcal{B}(
\mathbb{R}), \ \Phi_{f}(t,s,x)v=\frac{f(s)}{f(t)}e^{(t-s)}v.
\]
Then $C_{f}=(\varphi,\Phi_{f})$ is a skew-evolution semiflow on
$Y=\mathbb{R}_{+}\times \mathbb{R}$ over every evolution semiflow
$\varphi$ on $\mathbb{R}_{+}$. As
\[
\left | \Phi_{f}(t,s,x)v\right |\geq \frac{1}{f(t)}e^{(t-s)}|v|, \
\forall (t,s,x,v)\in T\times Y,
\]
it follows that $C_{f}$ is exponentially instable.
\end{example}

\begin{definition}\rm\label{def_nedc}
A skew-evolution semiflow $C$ has \emph{exponential decay} if
there exist the mappings $M$,
$\omega:\mathbb{R}_{+}\rightarrow\mathbb{R}_{+}^{*}$ such that
\begin{equation}\label{rel_nedc}
\left\Vert \Phi(s,t_{0},x)v\right\Vert \leq M(t)e^{\omega(t)
(t-s)}\left\Vert \Phi(t,t_{0},x)v\right\Vert,
\end{equation}%
for all $(t,s),(s,t_{0})\in  T$ and all $(x,v)\in Y$.
\end{definition}

In order to describe the instability of skew-evolution semiflows
in discrete time, we will consider first the next

\begin{proposition}\label{caract_eis_discret}
A skew-evolution semiflow with exponential decay $C$ is
exponentially instable if and only if there exist a constant $\mu>
0$ and a sequence of real numbers $(a_{n})_{n\geq 0}$ with the
property $a_{n}\geq 1, \ \forall n\geq 0$ such that
\begin{equation}
\left\Vert\Phi(n,n_{0},x)v\right\Vert\leq
a_{m}e^{-\mu(m-n_{0})}\left\Vert\Phi(m,n_{0},x)v\right\Vert,
\end{equation}
for all $(m,n),(n,n_{0})\in\Delta$ and all $(x,v)\in Y$.
\end{proposition}

\begin{proof}
\emph{Necessity}. It is obtained if we consider in relation
$(\ref{rel_exp_instab})$ $t=m$, $s=n$, $t_{0}=n_{0}$ and if we
define
\[
a_{m}=N(m), \ m\in\mathbb{N} \ \textrm{\c{s}i} \ \mu=\nu>0,
\]
where the existence of function
$N:\mathbb{R}_{+}\rightarrow\mathbb{R}_{+}^{*}$ and of constant
$\nu$ is given by Definition \ref{cis}.

\emph{Sufficiency}. As a first step, we consider $t\geq t_{0}+1$
and we denote $n=[t]$ respectively $n_{0}=[t_{0}]$. We obtain
\[
n\leq t<n+1, \ n_{0}\leq t_{0}<n_{0}+1, \ n_{0}+1\leq n.
\]
It follows that
\[
\left\Vert\Phi(t,t_{0},x)v\right\Vert=
\]
\[
=\left\Vert\Phi(t,n,\varphi(n,n_{0}+1,x))\Phi(n,n_{0}+1,\varphi(n_{0}+1,t_{0},x))\Phi(n_{0}+1,t_{0},x)v\right\Vert\geq
\]
\[
\geq
[M(t)]^{-1}e^{-\omega(t-n)}[M(t)]^{-1}e^{-\omega(n_{0}+1-t_{0})}\left\Vert\Phi(n,n_{0}+1,x)v\right\Vert\geq
\]
\[
\geq
\frac{[M(t)]^{-2}}{\widetilde{N}}e^{-2\omega}e^{\widetilde{\nu}(n-n_{0}+1)}\left\Vert
v\right\Vert\geq
\frac{e^{\widetilde{\nu}}}{[M(t)]^{2}\widetilde{N}e^{2\omega}}e^{\widetilde{\nu}(t-t_{0})}\left\Vert
v\right\Vert,
\]
for all $(x,v)\in Y$, where the existence of function $M$ and of
constant $\omega$ is assured by Definition \ref{def_nedc}.

As a second step, if we consider $t\in[t_{0},t_{0}+1)$ we obtain
\[
M\left\Vert\Phi(t,t_{0},x)v\right\Vert\geq
e^{-\omega(t-t_{0})}\left\Vert v\right\Vert\geq
e^{-(\widetilde{\nu}+\omega)}e^{\widetilde{\nu}(t-t_{0})}\left\Vert
v\right\Vert,
\]
for all $(x,v)\in Y$.

Hence
\[
N\left\Vert\Phi(t,t_{0},x)v\right\Vert\geq
e^{\nu(t-t_{0})}\left\Vert v\right\Vert,
\]
for all $(t,t_{0},x,v)\in  T\times Y$, where we have denoted
\[
N=Me^{(\widetilde{\nu}+\omega)}+M^{2}\widetilde{N}e^{(-\widetilde{\nu}+2\omega)}
\ \textrm{\c{s}i} \ \nu=\widetilde{\nu},
\]
which proves the exponential instability of $C$.
\end{proof}

\begin{theorem}\label{th_D_is_discret_neunif}
A skew-evolution semiflow $C$ is exponentially instable if and
only if there exist a function $R\in\mathcal{R}$, a constant
$\rho<0$ and a sequence of real numbers $(\alpha_{n})_{n\geq 0}$
with the property $\alpha_{n}\geq 1, \ \forall n\geq 0$ such that
\begin{equation}
\sum_{k=n}^{m}R\left(e^{-\rho(m-k)}\left\Vert
\Phi(k,n,x)v\right\Vert\right)\leq R\left(\alpha_{m}\left\Vert
\Phi(m,n,x)v\right\Vert\right),
\end{equation}
for all $(m,n)\in\Delta$ and all $(x,v)\in Y$.
\end{theorem}

\begin{proof} \emph{Necessity}. Let $R(t)=t, \ t\geq 0$.
Proposition \ref{caract_eis_discret} assures the existence of a
constant $\nu> 0$ and of a sequence of real numbers
$(a_{n})_{n\geq 0}$ with the property $a_{n}\geq 1, \ \forall
n\geq 0$. We obtain for $\rho=-\frac{\nu}{2}>0$
\[
\sum_{k=n}^{m}e^{-\rho(m-k)}\left\Vert\Phi(k,n,x)v\right\Vert \leq
a_{m} \sum_{k=n}^{m}e^{-\rho(m-k)}e^{-\nu(m-k)}\left\Vert
\Phi(m,n,x)v\right\Vert =
\]
\[
=a_{m}\left\Vert \Phi(m,n,x)v\right\Vert
\sum_{k=n}^{m}e^{-\frac{\nu}{2}(k-n)}\leq \alpha_{m}\left\Vert
\Phi(m,n,x)v\right\Vert,
\]
for all $(m,n) \in \Delta$ and all $(x,v)\in Y$, where we have
denoted
\[
\alpha_{m}=\frac{1}{1-e^{\rho}}a_{m}, \ m\in \mathbb{N}.
\]

\emph{Sufficiency}. According to the hypothesis, if we consider
$k=n$ we obtain
\[
R\left(e^{-\rho(m-n)}\left\Vert\Phi(n,n,x)v\right\Vert\right)\leq
R\left(a_{m}\left\Vert \Phi(m,n,x)v\right\Vert\right),
\]
for all $(m,n,x,v)\in\Delta\times Y$, which implies, by means of
the properties of function $R$, the exponential instability of $C$
and ends the proof.
\end{proof}

\vspace{3mm}

For $R(t)=t^{p}$, $t\geq 0$, $p>0$ it is obtained

\begin{corollary}\label{cor_D_is_discret_neunif}
Let $p>0$. A skew-evolution semiflow $C$ is exponentially instable
if and only if there exist a constant $\rho<0$ and a sequence of
real numbers $(\alpha_{n})_{n\geq 0}$ with the property
$\alpha_{n}\geq 1, \ \forall n\geq 0$ such that
\begin{equation}
\sum_{k=n}^{m} e^{-p\rho(m-k)}\left\Vert
\Phi(k,n,x)v\right\Vert^{p}\leq  \alpha_{m}\left\Vert
\Phi(m,n,x)v\right\Vert^{p},
\end{equation}
for all $(m,n)\in\Delta$ and all $(x,v)\in Y$.
\end{corollary}

\section{Nonuniform discrete exponential dichotomy}

\begin{definition}\rm\label{prinv}
A projector $P$ on $Y$ is called \emph{invariant} relative to a
skew-evolution semiflow $C=(\varphi,\Phi)$ if following relation
\begin{equation}
P(\varphi(t,s,x))\Phi(t,s,x)=\Phi(t,s,x)P(x),
\end{equation}
holds for all $(t,s)\in  T$ and all $x\in  X$.
\end{definition}

\begin{definition}\rm\label{comp_pr_dich}
Two projectors $P_{1}$ and $P_{2}$ are said to be
\emph{compatible} with a skew-evolution semiflow
$C=(\varphi,\Phi)$ if

$(d_{1})$ projectors $P_{1}$ and $P_{2}$ are invariant on $Y$;

$(d_{2})$ for all $x\in  X$, the projections $P_{1}(x)$ and
$P_{2}(x)$ verify the relations
\[
P_{1}(x)+P_{2}(x)=I \ \textrm{and}\
P_{1}(x)P_{2}(x)=P_{2}(x)P_{1}(x)=0.
\]
\end{definition}

\begin{definition}\rm\label{def_d_ed}
A skew-evolution semiflow $C=(\varphi,\Phi)$ is called

$(d)$ \emph{dichotomic} if there exist some functions $N_{1}$,
$N_{2}:\mathbb{R}_{+}\rightarrow \mathbb{R}_{+}^{\ast }$, two
projectors $P_{1}$ and $P_{2}$ compatible with $C$ such that
\begin{equation}\label{d_stab}
\left\Vert \Phi(t,t_{0},x)P_{1}(x)v\right\Vert \leq
N_{1}(s)\left\Vert \Phi(s,t_{0},x)P_{1}(x)v\right\Vert
\end{equation}
\begin{equation}\label{d_instab}
\left\Vert \Phi(s,t_{0},x)P_{2}(x)v\right\Vert \leq
N_{2}(t)\left\Vert \Phi(t,t_{0},x)P_{2}(x)v\right\Vert,
\end{equation}
for all $(t,s),(s,t_{0})\in T$ and all $(x,v)\in Y$;

$(ed) $\emph{exponentially dichotomic} if there exist functions
$N_{1}$, $N_{2}:\mathbb{R}_{+}\rightarrow \mathbb{R}_{+}^{\ast }$,
constants $\nu_{1}$, $\nu_{2}>0$ and two projectors $P_{1}$ and
$P_{2}$ compatible with $C$ such that
\begin{equation}\label{dich_stab}
e^{\nu_{1}(t-s)}\left\Vert \Phi(t,t_{0},x)P_{1}(x)v\right\Vert
\leq N_{1}(s)\left\Vert \Phi(s,t_{0},x)P_{1}(x)v\right\Vert
\end{equation}
\begin{equation}\label{dich_instab}
e^{\nu_{2}(t-s)}\left\Vert \Phi(s,t_{0},x)P_{2}(x)v\right\Vert
\leq N_{2}(t)\left\Vert \Phi(t,t_{0},x)P_{2}(x)v\right\Vert
\end{equation}
for all $(t,s),(s,t_{0})\in T$ and all $(x,v)\in Y$.
\end{definition}

\begin{remark}\rm
An exponentially dichotomic skew-evolution semiflow is dichotomic.
\end{remark}

\begin{example}\rm\label{ex_nued}
Let $ X=\mathbb{R}_{+}$ and $V=\mathbb{R}^{2}$ endowed with the
norm
\begin{equation*}
\left\Vert(v_{1},v_{2})\right\Vert=|v_{1}|+|v_{2}|, \
v=(v_{1},v_{2})\in V.
\end{equation*}
The mapping $\Phi: T\times  X\rightarrow \mathcal{B}(V)$, given by
\begin{equation*}
\Phi(t,s,x)(v_{1},v_{2})=(e^{t\sin t-s\sin s-2t+2s
}v_{1},e^{2t-2s-3t\cos t+3s\cos s}v_{2})
\end{equation*}
is an evolution cocycle over every evolution semiflow $\varphi$.
We consider the projectors
\begin{equation*}
P_{1}(x)(v_{1},v_{2})=(v_{1},0) \ \textrm{and} \
P_{2}(x)(v_{1},v_{2})=(0,v_{2}).
\end{equation*}
As
\[
t\sin t-s\sin s-2t+2s\leq -t+3s, \ \forall (t,s)\in T,
\]
we obtain that
\[
\left\Vert\Phi(t,s,x)P_{1}(x)v\right\Vert\leq e^{2s}e^{-(t-s)}|
v_{1}|,\ \forall (t,s,x,v)\in T\times Y.
\]
Similarly, as
\[
2t-2s-3t\cos t+3s\cos s\geq -t-5s, \ \forall (t,s)\in T
\]
it follows that
\[
e^{6t}\left \Vert\Phi(t,s,x)P_{2}(x)v\right\Vert \geq
e^{5(t-s)}|v_{2}|,\ \forall (t,s,x,v)\in T\times Y.
\]
The skew-evolution semiflow $C=(\varphi,\Phi)$ is exponentially
dichotomic with characteristics
\begin{equation*}
N(u)=e^{6u} \ \textrm{and} \  \nu=1.
\end{equation*}
\end{example}

\vspace{3mm}

In what follows, let us denote
\[
C_{k}(t,s,x,v)=(\varphi(t,s,x),\Phi_{k}(t,s,x)v), \ \forall
(t,t_{0},x,v)\in  T\times Y, \ \forall k\in \{1,2\},
\]
where
\[
\Phi_{k}(t,t_{0},x)=\Phi(t,t_{0},x)P_{k}(x), \ \forall
(t,t_{0})\in  T, \ \forall x\in  X, \ \forall k\in \{1,2\}.
\]

In discrete time, we will describe the property of exponential
dichotomy as given in the next

\begin{proposition} \label{d_trdscr}
A skew-evolution semiflow $C =(\varphi,\Phi)$ is exponentially
dichotomic if and only if there exist two projectors
$\{P_{k}\}_{k\in \{1,2\}}$, compatible with $C$, constants
$\nu_{1}\leq 0\leq \nu_{2}$ and a sequence of real positive
numbers $(a_{n})_{n\geq 0}$ such that

$(d_{1}')$
\begin{equation}
\left\Vert \Phi(m,n,x)P_{1}(x)v\right\Vert \leq a_{n}\left\Vert
P_{1}(x)v\right\Vert e^{\nu_{1}(m-n)}
\end{equation}

$(d_{2}')$
\begin{equation}
\left\Vert P_{2}(x)v\right\Vert \leq a_{m}\left\Vert
\Phi(m,n,x)P_{2}(x)v\right\Vert e^{-\nu_{2}(m-n)}
\end{equation}
for all $(m,n)\in \Delta$ and all $(x,v)\in Y$.
\end{proposition}

\begin{proof} \emph{Necessity}. If we consider for $C_{1}$
in relation $(\ref{dich_stab})$ of Definition \ref{def_d_ed}
$t=m$, $s=t_{0}=n$ and if we define
\[
a_{n}=N(n), \ n\in\mathbb{N},
\]
relation $(d_{1}')$ is obtained.

Statement $(d_{2})'$ results from Definition \ref{def_d_ed} for
$C_{2}$ if we consider in relation $(\ref{dich_instab})$ $t=m$,
$s=t_{0}=n$ and
\[
a_{m}=N(m), \ m\in\mathbb{N}.
\]
\emph{Sufficiency}. It is obtained by means of Proposition
\ref{caract_es_discret} for $C_{1}$, respectively of Proposition
\ref{caract_eis_discret} for $C_{2}$.

Hence, $C$ is exponentially dichotomic.
\end{proof}

\begin{theorem}
A skew-evolution semiflow $C =(\varphi,\Phi)$ is exponentially
dichotomic if and only if there exist two projectors
$\{P_{k}\}_{k\in \{1,2\}}$, compatible with $C$ such that

$(ed_{1}')$ there exist a constant $\rho_{1}>0$ and a sequence of
real positive numbers $(\alpha_{n})_{n\geq 0}$ such that
\begin{equation}
\sum_{k=n}^{m}e^{\rho_{1}(k-n)}\left\Vert
\Phi(k,n,x)P_{1}(x)v\right\Vert \leq \alpha_{n}\left\Vert
P_{1}(x)v\right\Vert
\end{equation}

$(ed_{2}')$ there exist a constant $\rho_{2}<0$ and a sequence of
real positive numbers $(\beta_{n})_{n\geq 0}$ such that
\begin{equation}
\sum_{k=n}^{m}e^{-\rho_{2}(m-k)}\left\Vert
\Phi(k,n,x)P_{2}(x)v\right\Vert \leq\beta_{m}\left\Vert
\Phi(m,n,x)P_{2}(x)v\right\Vert
\end{equation}
for all $(m,p),(p,n)\in \Delta$ and all $(x,v)\in Y$.
\end{theorem}

\begin{proof}
It is obtained by means of Theorem \ref{th_D_discret_neunif} and
Theorem \ref{th_D_is_discret_neunif}, by considering $R(t)=t, \
t\geq 0$.
\end{proof}

\section{Nonuniform discrete exponential trichotomy}

\begin{definition}\rm\label{3pr}
Three projectors $\{P_{k}\}_{k\in \{1,2,3\}}$ are said to be
\emph{compatible} with a skew-evolution semiflow
$C=(\varphi,\Phi)$ if

$(t_{1})$ each projector $P_{k}$, $k\in \{1,2,3\}$ is invariant on
$Y$;

$(t_{2})$ for all $x\in  X$, the projections $P_{0}(x)$,
$P_{1}(x)$ and $P_{2}(x)$ verify the relations
\begin{equation*}
P_{1}(x)+P_{2}(x)+P_{3}(x)=I \ \textrm{and} \ P_{i}(x)P_{j}(x)=0,
\ \forall i,j \in \{1,2,3\}, \ i\neq j.
\end{equation*}
\end{definition}

\begin{definition}\rm\label{net}
A skew-evolution semiflow $C=(\varphi,\Phi)$ is called

$(t)$ \emph{trichotomic} if there exist the functions $N_{1}$,
$N_{2}$, $N_{3}:\mathbb{R}_{+}\rightarrow \mathbb{R}_{+}^{\ast }$
and three projectors $P_{1}$, $P_{2}$ and $P_{3}$ compatible with
$C$ such that
\begin{equation}
\left\Vert \Phi(t,t_{0},x)P_{1}(x)v\right\Vert \leq
N_{1}(s)\left\Vert \Phi(s,t_{0},x)P_{1}(x)v\right\Vert
\end{equation}
\begin{equation}
\left\Vert \Phi(s,t_{0},x)P_{2}(x)v\right\Vert \leq
N_{2}(t_{0})\left\Vert \Phi(t,t_{0},x)P_{2}(x)v\right\Vert
\end{equation}
\begin{equation}
\left\Vert \Phi(t,t_{0},x)P_{3}(x)v\right\Vert\leq
N_{3}(s)\left\Vert \Phi(s,t_{0},x)P_{3}(x)v\right\Vert
\end{equation}
\begin{equation}
\left\Vert \Phi(s,t_{0},x)P_{3}(x)v\right\Vert \leq
N_{3}(t)\left\Vert \Phi(t,t_{0},x)P_{3}(x)v\right\Vert
\end{equation}
for all $(t,s),(s,t_{0})\in T$ and all $(x,v)\in Y$;

$(et)$ \emph{exponentially trichotomic} if there exist the
functions $N_{1}$, $N_{2}$, $N_{3}$,
$N_{4}:\mathbb{R}_{+}\rightarrow \mathbb{R}_{+}^{\ast }$,
constants $\nu_{1}$, $\nu_{2}$, $\nu_{3}$, $\nu_{4}$ with the
properties
\[
\nu_{1}\leq\nu_{2}\leq 0\leq\nu_{3}\leq\nu_{4}
\]
and three projectors $P_{1}$, $P_{2}$ and $P_{3}$ compatible with
$C$ such that
\begin{equation}\label{Pes}
\left\Vert \Phi(t,t_{0},x)P_{1}(x)v\right\Vert \leq
N_{1}(s)\left\Vert \Phi(s,t_{0},x)P_{1}(x)v\right\Vert
e^{\nu_{1}(t-s)}
\end{equation}
\begin{equation}\label{Qeis}
N_{4}(t)\left\Vert \Phi(t,t_{0},x)P_{2}(x)v\right\Vert \geq
\left\Vert \Phi(s,t_{0},x)P_{2}(x)v\right\Vert e^{\nu_{4}(t-s)}
\end{equation}
\begin{equation}\label{Redc}
N_{2}(t)\left\Vert \Phi(t,t_{0},x)P_{3}(x)v\right\Vert \geq
\left\Vert \Phi(s,t_{0},x)P_{3}(x)v\right\Vert e^{\nu_{2}(t-s)}
\end{equation}
\begin{equation}\label{Reg}
\left\Vert \Phi(t,t_{0},x)P_{3}(x)v\right\Vert \leq
N_{3}(s)\left\Vert \Phi(s,t_{0},x)P_{3}(x)v\right\Vert
e^{\nu_{3}(t-s)}
\end{equation}
for all $(t,s),(s,t_{0})\in T$ and all $(x,v)\in Y$.
\end{definition}

\begin{remark}\rm
An exponentially trichotomic skew-evolution semiflow is
trichotomic.
\end{remark}

\begin{example}\rm\label{ex_nuet}
Let us consider the example of $C$ given in Example \ref{ex_ce}.
We consider the projections
\begin{equation*}
P_{1}(x)(v)=(v_{1},0,0), \ P_{2}(x)(v)=(0,v_{2},0), \
P_{3}(x)(v)=(0,0,v_{3}).
\end{equation*}
The skew-evolution semiflow $C=(\varphi,\Phi)$ is exponentially
trichotomic with characteristics
\begin{equation*}
\nu_{1}=\nu_{2}=-x(0),  \ \nu_{3}=x(0) \ \textrm{and} \ \nu_{4}=1,
\end{equation*}
\begin{equation*}
N_{1}(u)=e^{ux(0)}, \ \ N_{2}(u)=e^{-2lu}, \ \,
N_{3}(u)=e^{2ux(0)} \ \textrm{and} \ N_{4}(u)=e^{-lu}.
\end{equation*}
\end{example}

As in the case of dichotomy, we denote
\[
C_{k}(t,s,x,v)=(\varphi(t,s,x),\Phi_{k}(t,s,x)v), \ \forall
(t,t_{0},x,v)\in  T\times Y, \ \forall k\in \{1,2,3\},
\]
where
\[
\Phi_{k}(t,t_{0},x)=\Phi(t,t_{0},x)P_{k}(x), \ \forall
(t,t_{0})\in  T, \ \forall x\in  X, \ \forall k\in \{1,2,3\}.
\]

\vspace{3mm}

In discrete time, the trichotomy of a skew-evolution semiflow can
be described as in the next

\begin{proposition} \label{d_trdscr2}
A skew-evolution semiflow $C =(\varphi,\Phi)$ is exponentially
trichotomic if and only if there exist three projectors
$\{P_{k}\}_{k\in \{1,2,3\}}$ compatible with $C$, constants
$\nu_{1}$, $\nu_{2}$, $\nu_{3}$, $\nu_{4}$ with the property
$\nu_{1}\leq \nu_{2}\leq 0\leq \nu_{3}\leq \nu_{4}$ and a sequence
of positive real numbers $(a_{n})_{n\geq 0}$ such that

$(t_{1})$
\begin{equation}\label{rel1_trdscr2}
\left\Vert \Phi(m,n,x)P_{1}(x)v\right\Vert \leq a_{p}\left\Vert
\Phi(p,n,x)P_{1}(x)v\right\Vert e^{\nu_{1}(m-p)}
\end{equation}

$(t_{2})$
\begin{equation}\label{rel2_trdscr2}
\left\Vert \Phi(p,n,x)P_{2}(x)v\right\Vert \leq a_{m}\left\Vert
\Phi(m,n,x)P_{2}(x)v\right\Vert e^{-\nu_{4}(m-p)}
\end{equation}

$(t_{3})$
\begin{equation}\label{rel3_trdscr2}
\left\Vert \Phi(p,n,x)P_{3}(x)v\right\Vert \leq a_{m}\left\Vert
\Phi(m,n,x)P_{3}(x)v\right\Vert e^{-\nu_{2}(m-p)}
\end{equation}

$(t_{4})$
\begin{equation}\label{rel4_trdscr2}
a_{p} \left\Vert \Phi(p,n,x)P_{3}(x)v\right\Vert \geq \left\Vert
\Phi(m,n,x)P_{3}(x)v\right\Vert e^{-\nu_{3}(m-p)}
\end{equation}
for all $(m,p),(p,n)\in \Delta$ and all $(x,v)\in Y$.
\end{proposition}

\begin{proof} \emph{Necessity}. $(t_{1})$ is obtain if we consider
for $C_{1}$ in relation $(\ref{rel_exp_stab})$ of Definition
\ref{def_ns_nes} $t=n$, $s=t_{0}=m$ and if we define
\[
a_{p}=N(p), \ p\in\mathbb{N} \ \textrm{and} \ \nu_{1}=-\nu<0.
\]
$(t_{2})$ follows according to Definition \ref{cis} for $C_{2}$ if
we consider in relation $(\ref{rel_exp_instab})$ $t=m$, $s=n$,
$t_{0}=n_{0}$ and
\[
a_{m}=N(m), \ m\in\mathbb{N} \ \textrm{and} \ \nu_{4}=\nu>0.
\]
$(t_{3})$ is obtained for $C_{3}$ out of relatiom
$(\ref{rel_neg})$ of Definition \ref{def_neg} for $t=m$, $s=p$,
$t_{0}=n$ and if we define
\[
a_{m}=M(m)  \ \textrm{and} \ \nu_{2}=-\omega(m)<0, \
m\in\mathbb{N}.
\]
$(t_{4})$ follows for $C_{3}$ from relation $(\ref{rel_nedc})$ of
Definition \ref{def_nedc} for $t=m$, $s=p$, $t_{0}=n$ and if we
consider
\[
a_{p}=M(p) \ \textrm{and} \ \nu_{3}=\omega(p)>0, \ p\in\mathbb{N}.
\]
\emph{Sufficiency}. Let $t\geq t_{0}+1$. We denote $n=[t]$,
$n_{0}=[t_{0}]$ and we obtained the relations
\[
n\leq t<n+1, \ n_{0}\leq t_{0}<n_{0}+1, \ n_{0}+1\leq n.
\]
According to $(t_{1})$, we have
\[
\left\Vert\Phi_{1}(t,t_{0},x)v\right\Vert\leq
\]
\[
\leq
M(n)e^{\omega(n)(t-n)}\left\Vert\Phi(n,n_{0}+1,\varphi(n_{0}+1,t_{0},x))
\Phi(n_{0}+1,t_{0},x)P_{1}(x)v\right\Vert\leq
\]
\[
\leq a_{n}M^{2}(n)e^{2[\omega(n)+\mu]}e^{-\mu(t-t_{0})}\left\Vert
P_{1}(x)v\right\Vert,
\]
for all $(x,v)\in Y$, where functions $M$ and $\omega$ are given
as in Definition \ref{def_neg}.

For $t\in[t_{0},t_{0}+1)$ we have
\[
\left\Vert\Phi_{1}(t,t_{0},x)v\right\Vert\leq
M(t_{0})e^{\omega(t_{0})(t-t_{0})}\left\Vert
P_{1}(x)v\right\Vert\leq
\]
\[
\leq M(t_{0})e^{\omega(t_{0})+\mu} e^{-\mu(t-t_{0})}\left\Vert
P_{1}(x)v\right\Vert,
\]
for all $(x,v)\in Y$. Hence, relation $(\ref{Pes})$ is obtained.

Let $t\geq t_{0}+1$ and $n=[t]$ respectively $n_{0}=[t_{0}]$. It
follows
\[
n\leq t<n+1, \ n_{0}\leq t_{0}<n_{0}+1, \ n_{0}+1\leq n.
\]
From $(t_{2})$, it is obtained
\[
\left\Vert\Phi_{2}(t,t_{0},x)v\right\Vert=
\]
\[
=\left\Vert\Phi(t,n,\varphi(n,n_{0}+1,x))\Phi(n,n_{0}+1,
\varphi(n_{0}+1,t_{0},x))\Phi(n_{0}+1,t_{0},x)P_{2}(x)v\right\Vert\geq
\]
\[
\geq
[M(t)]^{-1}e^{-\omega(t-n)}[M(t)]^{-1}e^{-\omega(n_{0}+1-t_{0})}
\left\Vert\Phi_{2}(n,n_{0}+1,x)v\right\Vert\geq
\]
\[
\geq
\frac{[M(t)]^{-2}}{\widetilde{N}}e^{-2\omega}e^{\widetilde{\nu}(n-n_{0}+1)}\left\Vert
P_{2}(x)v\right\Vert\geq
\frac{e^{\widetilde{\nu}}}{[M(t)]^{2}\widetilde{N}e^{2\omega}}e^{\widetilde{\nu}(t-t_{0})}\left\Vert
P_{2}(x)v\right\Vert,
\]
for all $(x,v)\in Y$, where $M$ and $\omega$ are given by
Definition \ref{def_nedc}.

\noindent For $t\in[t_{0},t_{0}+1)$ we have
\[
M\left\Vert\Phi_{2}(t,t_{0},x)v\right\Vert\geq
e^{-\omega(t-t_{0})}\left\Vert P_{2}(x)v\right\Vert\geq
e^{-(\widetilde{\nu}+\omega)}e^{\widetilde{\nu}(t-t_{0})}\left\Vert
P_{2}(x)v\right\Vert,
\]
for all $(x,v)\in Y$. It follows that
\[
N\left\Vert\Phi_{2}(t,t_{0},x)v\right\Vert\geq
e^{\nu(t-t_{0})}\left\Vert P_{2}(x)v\right\Vert,
\]
for all $(t,t_{0},x,v)\in  T\times Y$, where we have denoted
\[
N=Me^{(\widetilde{\nu}+\omega)}+M^{2}\widetilde{N}e^{(-\widetilde{\nu}+2\omega)}
\ \textrm{and} \ \nu=\widetilde{\nu}
\]
and which implies relation $(\ref{Qeis})$.

By a similar reasoning, from $(t_{3})$ relation $(\ref{Redc})$ is
obtained, and from $(t_{4})$ relation $(\ref{Reg})$ follows.

Hence, the skew-evolution semiflow $C$ is exponentially
trichotomic.
\end{proof}

\vspace{3mm}

Some characterizations in discrete time for the exponential
trichotomy for skew-evolution semiflows are given in what follows.

\begin{theorem}\label{trich_discr_n}
A skew-evolution semiflow $C =(\varphi,\Phi)$ is exponentially
trichotomic if and only if there exist three projectors
$\{P_{k}\}_{k\in \{1,2,3\}}$ compatible cu $C$ such that $C_{1}$
has exponential growth, $C_{2}$ has exponential decay and such
that following relations hold

$(t_{1}')$ there exist a constant $\rho_{1}>0$ and a sequence of
positive real numbers $(\alpha_{n})_{n\geq 0}$ such that
\begin{equation}
\sum_{k=n}^{m}e^{\rho_{1}(k-n)}\left\Vert
\Phi(k,n,x)P_{1}(x)v\right\Vert \leq \alpha_{n}\left\Vert
P_{1}(x)v\right\Vert
\end{equation}

$(t_{2}')$ there exist a constant $\rho_{2}>0$ and a sequence of
positive real numbers $(\beta_{n})_{n\geq 0}$ such that
\begin{equation}
\sum_{k=n}^{m}e^{-\rho_{2}(k-n)}\left\Vert
\Phi(k,n,x)P_{2}(x)v\right\Vert
\leq\beta_{m}e^{-\rho_{2}(m-n)}\left\Vert
\Phi(m,n,x)P_{2}(x)v\right\Vert
\end{equation}

$(t_{3}')$ there exist a constant $\rho_{3}>0$ and a sequence of
positive real numbers $(\gamma_{n})_{n\geq 0}$ such that
\begin{equation}
\sum_{k=p}^{m}e^{-\rho_{3}(k-n)}\left\Vert
\Phi(k,n,x)P_{3}(x)v\right\Vert \leq
\gamma_{n}e^{-\rho_{3}(p-n)}\left\Vert\Phi(p,n,x)
P_{3}(x)v\right\Vert
\end{equation}

$(t_{4}')$ there exist a constant $\rho_{4}>0$ and a sequence of
positive real numbers $(\delta_{n})_{n\geq 0}$ such that
\begin{equation}
\sum_{k=p}^{m}e^{\rho_{4}(k-n)}\left\Vert
\Phi(k,n,x)P_{3}(x)v\right\Vert \leq
\delta_{m}e^{\rho_{4}(m-n)}\left\Vert \Phi(m,n,x)
P_{3}(x)v\right\Vert
\end{equation}
for all $(m,p),(p,n)\in \Delta$ and all $(x,v)\in Y$.
\end{theorem}

\begin{proof}
$(t_{1})\Leftrightarrow (t_{1}')$ \emph{Necessity}. As $C$ is
exponentially trichotomic, Proposition \ref{d_trdscr2} assures the
existence of a projector $P$, of a constant $\nu_{1}\leq 0 $ and
of a sequence of positive real numbers $(a_{n})_{n\geq 0}$ such
$(t_{1})$ holds. Let $P_{1}=P$. If we consider
\[
\rho_{1}=-\frac{\nu_{1}}{2}>0,
\]
we obtain
\begin{equation*}
\sum_{k=n}^{m}e^{\rho_{1}(k-n)}\left\Vert\Phi(k,n,x)P_{1}(x)v\right\Vert\leq
\end{equation*}
\begin{equation*}
\leq a_{n}
\sum_{k=n}^{m}e^{\rho_{1}(k-n)}e^{\nu_{1}(k-n)}\left\Vert
\Phi(n,n,x)P_{1}(x)v\right\Vert =
\end{equation*}
\begin{equation*}
=a_{n}\left\Vert P_{1}(x)v\right\Vert
\sum_{k=n}^{m}e^{(\rho_{1}+\nu_{1})(k-n)}\leq \alpha_{n}\left\Vert
P_{1}(x)v\right\Vert, \forall m,n \in \mathbb{N}, \forall (x,v)\in
Y,
\end{equation*}
where we have denoted
\[
\alpha_{n}=a_{n}e^{-\frac{\nu_{1}}{2}}, \ n\in \mathbb{N}.
\]
\emph{Sufficiency}. As $C_{1}$ has exponential growth, there exist
constants $M\geq 1$ and $r>1$ such that relation
\begin{equation}
\left\Vert \Phi(n+p,n,x)v\right\Vert \leq Mr^{p}\left\Vert
v\right\Vert,
\end{equation}
holds for all $n,p\in  \mathbb{N}$ and all $(x,v)\in Y$. If we
denote $\omega=\ln r>0$ the inequality can be written as follows
\begin{equation*}
\left\Vert \Phi(m,n,x)\right\Vert \leq Me^{\omega(m-k)}\left\Vert
\Phi(k,n,x)\right\Vert, \ \forall (m, k), (k,n)\in \Delta, \ x\in
 X.
\end{equation*}
We consider successively the $m-n+1$ relations. By denoting
$P=P_{1}$, we obtain
\begin{equation*}
\left\Vert \Phi(m,n,x)P(x)v\right\Vert
\leq\frac{Me^{\omega(m-n)}}{1+e^{\rho_{1}}+...+e^{\rho_{1}(m-n)}}\alpha_{n}\left\Vert
P(x)v\right\Vert,
\end{equation*}
for all $(m, n)\in \Delta$ and all $(x,v)\in Y$. If we define the
constant
\[
\nu_{1}=\omega -\rho_{1} \ \textrm{pentru} \ \omega <\rho_{1}
\]
and the sequence of nonnegative real numbers
\[
a_{n}=M\alpha_{n}, \ n\in\mathbb{N},
\]
relation $(t_{1})$ is obtained.

Similarly can also be proved the other equivalences.
\end{proof}

\vspace{3mm}

In order to characterize the exponential trichotomy by means of
four projectors, we give the next

\begin{definition}\rm
Four invariant projectors $\{R_{k}\}_{k\in \{1,2,3,4\}}$ that
satisfies for all $(x,v)\in Y$ following relations:

$(pc_{1}')$ $%
R_{1}(x)+R_{3}(x)=R_{2}(x)+R_{4}(x)=I$;

$(pc_{2}')$ $%
R_{1}(x)R_{2}(x)=R_{2}(x)R_{1}(x)=0$ \textrm{\c{s}i} $%
R_{3}(x)R_{4}(x)=R_{4}(x)R_{3}(x)$;

$(pc_{3}')$ $\left\Vert \left[
R_{1}(x)+R_{2}(x)\right]v\right\Vert ^{2}=\left\Vert
R_{1}(x)v\right\Vert ^{2}+\left\Vert R_{2}(x)v\right\Vert ^{2}$;

$(pc_{4}')$ $\left\Vert \left[ R_{1}(x)+R_{3}(x)R_{4}(x)\right]
v\right\Vert ^{2}=\left\Vert R_{1}(x)v\right\Vert ^{2}+\left\Vert
R_{3}(x)R_{4}(x)v\right\Vert ^{2}$;

$(pc_{5}')$ $\left\Vert \left[R_{2}(x)+R_{3}(x)R_{4}(x)\right]
v\right\Vert ^{2}=\left\Vert R_{2}(x)v\right\Vert ^{2}+\left\Vert
R_{3}(x)R_{4}(x)v\right\Vert ^{2}$,

\noindent are called \emph{compatible} with the skew-evolution
semiflow $C$.
\end{definition}

\begin{theorem}\label{th_discret_4pr}
A skew-evolution semiflow $C=(\varphi,\Phi)$ is exponentially
trichotomic if and only if there exist four projectors
$\{R_{k}\}_{k\in \{1,2,3,4\}}$ compatible with $C$, constants $\mu
> \nu
>0$ and a sequence of positive real numbers $(\alpha_{n})_{n\geq 0}$
such that

$(t_{1}'')$
\begin{equation}\label{rel1_th_d_4pr}
\left\Vert \Phi(m+p,m,x)R_{1}(x)v\right\Vert \leq
\alpha_{m}\left\Vert R_{1}(x)v\right\Vert e^{-\nu p}
\end{equation}

$(t_{2}'')$
\begin{equation}\label{rel2_th_d_4pr}
\left\Vert R_{2}(x)v\right\Vert \leq \alpha_{p}\left\Vert
\Phi(m+p,m,x)R_{2}(x)v\right\Vert e^{-\mu p}
\end{equation}

$(t_{3}'')$
\begin{equation}\label{rel3_th_d_4pr}
\left\Vert R_{3}(x)v\right\Vert \leq \alpha_{p}\left\Vert
\Phi(m+p,m,x)R_{3}(x)v\right\Vert e^{\nu
 p}
\end{equation}

$(t_{4}'')$
\begin{equation}\label{rel4_th_d_4pr}
\left\Vert \Phi(m+p,m,x)R_{4}(x)v\right\Vert \leq
\alpha_{m}\left\Vert R_{4}(x)v\right\Vert e^{\mu p}
\end{equation}
for all $m,p\in \mathbb{N}$ and all $(x,v)\in Y$.
\end{theorem}

\begin{proof}\emph{Necessity}.  As $C$ is exponentially
trichotomic, according to Proposition \ref{d_trdscr2} there exist
three projectors $\{P_{k}\}_{k\in \{1,2,3\}}$ compatible with $C$,
constants $\nu_{1}\leq \nu_{2}\leq 0\leq \nu_{3}\leq \nu_{4}$ and
a sequence of positive real numbers $(a_{n})_{n\geq 0}$ such that
relations $(\ref{rel1_trdscr2})$--$(\ref{rel4_trdscr2})$ hold.

We will define the projectors
\[
R_{1}=P_{1}, \ R_{2}=P_{2}, \ R_{3}=I-P_{1} \ \textrm{and} \
R_{4}=I-P_{2},
\]
such that
\[
R_{3}R_{4}=R_{4}R_{3}=P_{3}.
\]
Projectors $R_{1}$, $R_{2}$, $R_{3}$ \c{s}i $R_{4}$ are compatible
with $C$. Let us define
\[
\mu=\nu_{3}=\nu_{4}>0, \ \nu=-\nu_{1}=-\nu_{2}>0 \ \textrm{and} \
\alpha_{n}=a_{n}, \ n\in\mathbb{N}.
\]
Hence, relations $(\ref{rel1_th_d_4pr})$--$(\ref{rel4_th_d_4pr})$
hold.

\emph{Sufficiency}. We consider the projectors
\[
P_{1}=R_{1}, \ P_{2}=R_{2}, \ P_{3}=R_{3}R_{4}.
\]
These are compatible with $C$.

The statements of Proposition \ref{d_trdscr2} follow if we
consider
\[
\nu_{1}=\nu_{2}=-\nu<0, \ \nu_{3}=\nu_{4}=\mu>0 \  \textrm{and} \
a_{n}=\alpha_{n}, \ n\in\mathbb{N}.
\]
Hence, $C$ is exponentially trichotomic, which ends the proof.
\end{proof}

{\footnotesize

\vspace{5mm}

\noindent\begin{tabular}[t]{ll}

Codru\c{t}a Stoica \\
Institut de Math\' ematiques  \\
Universit\' e Bordeaux 1, France  \\
e-mail: \texttt{codruta.stoica@math.u-bordeaux1.fr}\\

\vspace{3mm}\\

Mihail Megan \\
Faculty of Mathematics and Computer Science \\
West University of Timi\c{s}oara \\
email: \texttt{megan@math.uvt.ro}
\end{tabular}}

\end{document}